\documentclass[12pt, reqno ]{amsart}
\usepackage{amsmath,amssymb,amsthm,verbatim,color}
\usepackage[english]{babel}
\usepackage{amsfonts}
\usepackage{amssymb}
\usepackage{graphicx} 
\usepackage{amsmath}
\usepackage{mathtools}
\usepackage{amssymb}
\usepackage{amsthm}
\usepackage{float}
\usepackage{xcolor}

\usepackage[style=alphabetic, url=false, maxbibnames=10, doi=false]{biblatex}        
\newtheorem{theorem}{Theorem}[section]
\newtheorem{prop}[theorem]{Proposition}
\newtheorem{definition}[theorem]{Definition}

\newtheorem{lemma}[theorem]{Lemma}
\numberwithin{equation}{section}

\usepackage[colorlinks=true, allcolors=blue]{hyperref}

\title{Uniqueness in the Plateau Problem near Quadratic cones}
\author{Vishnu Nandakumaran}
\address{Department of Mathematics, University of Notre Dame, Notre Dame, IN, USA}
\email{vnandaku@nd.edu}

\author{G\'abor Sz\'ekelyhidi}
\address{Department of Mathematics,
Northwestern University,
Evanston, IL, USA}
\email{gaborsz@northwestern.edu}
\begin{document}

\begin{abstract}
We consider minimal hypersurfaces inside the unit ball whose boundary on the sphere is a small perturbation of the link of a minimizing quadratic cone. We show that such minimal surfaces are uniquely determined by their boundary condition. In particular the solutions of the Plateau problem are unique for  boundary conditions given by small perturbations of the link of a quadratic cone. 
\end{abstract}

\maketitle

\section{Introduction}

The classical Plateau problem asks for a surface of least area
spanning a given curve in $\mathbb{R}^3$. The existence of such a
surface was shown by Douglas~\cite{Douglas31} and Rad\'o~\cite{Rado30}, and a natural
question is to what extent solutions of Plateau's problem are
unique. Examples of Nitsche~\cite{Nitsche68} show that uniqueness does not
hold for all boundary curves, although it does hold for curves with
total curvature at most $4\pi$. Later, Morgan~\cite{Morgan78} showed that
for generic boundary curves in a suitable sense, Plateau's problem has a unique
solution.  Higher dimensional versions of this uniqueness theorem
were obtained by Morgan~\cite{Morgan82} in Euclidean space, while
Caldini-Marchese-Merlo-Steinbr\"uchel~\cite{CMMS24} extended this result
to a general setting of minimizing integral currents in Riemannian
manifolds.

In this paper we consider the question of the uniqueness of solutions
to Plateau's problem for some special class of boundaries in higher
dimensions. Suppose that $\mathcal{C}\subset \mathbb{R}^{n+1}$ is a
minimizing hypercone. It was shown by Hardt-Simon~\cite{SimonHardt} that if
$\mathcal{C}$ is smooth away from the origin, then there exists a
foliation of $\mathbb{R}^{n+1}\setminus \mathcal{C}$ by smooth minimal
hypersurfaces asymptotic to $\mathcal{C}$. This implies in particular,
that if we let $\Sigma = \mathcal{C} \cap \partial B_1$ be the link of
$\mathcal{C}$ in the unit sphere, then $\mathcal{C}\cap B_1$ is the
unique solution of Plateau's problem with boundary
$\Sigma$. A natural question is whether Plateau's
problem still has a unique solution if we slightly perturb the
boundary $\Sigma$, see for instance Wang~\cite[P4.1]{wang2020deformationssingularminimalhypersurfaces}.
This is closely related to the question of whether
in general the solutions of Plateau's problem are at least locally
unique.

Our main result is the following, answering this question in the case
of the simplest minimizing hypercones, namely the quadratic
cones. Note that the uniqueness holds for the more general class of minimal hypersurfaces with given boundary, rather than just area minimizers.  
\begin{theorem}\label{thm:main}
  Let $\mathcal{C} = C(S^p\times S^q)$ denote a minimizing quadratic
  cone. There is an $\epsilon > 0$ such that if $\Sigma'$ is the
  spherical graph of a function $g$ over $\Sigma = \mathcal{C}\cap
  \partial B_1$ with $\Vert g\Vert_{C^{2,\alpha}}  <\epsilon$, then
  there is a unique minimal hypersurface in $B_1$ with boundary
  given by $\Sigma'$. 
\end{theorem}

Here, by a minimal hypersurface we mean a codimension one stationary varifold (see Section~\ref{sec:prelim} for more details).
A natural approach to proving such a result would be to show that for
given $g$ as above, we can construct a minimal hypersurface $M_g$, which
moreover varies continuously as we vary $g$. Then one could hope that
the minimal hypersurfaces $M_{g + t}$ define a foliation as $t\in
(-\epsilon, \epsilon)$, and one could use these to show that $M_g$ is
a unique solution to the corresponding Plateau problem. A first step
towards this was achieved by the first author~\cite{nandakumaran2024minimalsurfacesnearhardtsimon},
constructing suitable $M_g$ using a gluing procedure that generalized
Caffarelli-Hardt-Simon's work~\cite{CHS}. Unfortunately it was not
clear from the construction whether the resulting hypersurfaces vary
continuously. Note, however, that using Theorem 1.1 we know that the minimal hypersurfaces constructed in \cite{nandakumaran2024minimalsurfacesnearhardtsimon} are the unique minimal hypersurfaces in $B_1$ with given boundary values close to $\Sigma$. This shows that up to a small translation and rotation, every such minimal hypersurface can be realized as a graph over $\mathcal{C}$, or over a scaling of the Hardt-Simon smoothings (see Theorem~\ref{thm:HardtSimon}). This gives a different approach to the result of Edelen-Spolaor~\cite{Nick}. 

In this paper, we pursue a different approach, arguing by
contradiction. First, we will restrict our attention to minimal hypersurfaces $M$ as in Theorem~\ref{thm:main}, with the additional condition that the total mass $\Vert M\Vert \leq \frac{3}{2}\Vert \mathcal{C}\cap B_1\Vert$, in order to rule out possible higher multiplicity. Then we show that if we had sequences $M^i_1,
M^i_2$ of minimal hypersurfaces in $B_1$ with this mass bound, whose boundaries were the graphs of $g_i$ over $\Sigma$ with $g_i \to 0$ in $C^{2,\alpha}$,  then either $M^i_1 = M^i_2$ for all sufficiently large
$i$, or we could extract one of the following:
\begin{itemize}
  \item A non-zero Jacobi field $v$ on $\mathcal{C}\cap B_1$ which vanishes
    along $\mathcal{C}\cap \partial B_1$, and in addition $|v(x)| =
    O(|x|^{\frac{2-n}{2}})$ as $|x|\to 0$.
  \item A non-zero Jacobi field $v$ on the Hardt-Simon foliate $S$
    such that $|v(x)| = O(|x|^{\frac{2-n}{2}})$ as $|x| \to \infty$.
  \item Two complete minimal hypersurfaces $S_1, S_2$ asymptotic to
    $\mathcal{C}$ at infinity, such that $S_2$ approaches $S_1$ at a
    rate faster than $|x|^{\frac{2-n}{2}}$ as $|x|\to\infty$.
\end{itemize}
Each of these three possibilities leads to a contradiction, which in
turn leads to the uniqueness statement of Theorem~\ref{thm:main}, for minimal hypersurfaces with the above mass bound. In order to remove the assumption of the mass bound, we will argue as outlined above with a family $M_{g+t}$ of area minimizers with boundary values given by $g+t$. The area minimizers satisfy the mass bound, and the uniqueness statement then implies that the $M_{g+t}$ vary continuously with $t$. See Section~\ref{sec:mainproof} for details.

It is an interesting and difficult problem to extend our result to
other types of minimal hypercones. In essence the crucial property of
quadratic cones that we use is that by Simon-Solomon~\cite{SimonSolomon} we
have a classification of all minimal hypersurfaces asymptotic to
quadratic cones. Note that Chan~\cite{Chan97} studied this kind of
classification problem for more general minimizing hypercones, and a
basic difficulty is a uniqueness question related to our work (see
\cite[Remark (1) on p. 180]{Chan97}).

\subsection*{Acknowledgements}
We would like to thank Nick Edelen and Zhihan Wang for helpful discussions. G. Sz. was supported in part by NSF grant DMS-2506325.

\section{Notation and Preliminaries}\label{sec:prelim}

\subsection{Varifolds} \label{sec:varifolds}
In this section we briefly recall some notation that we will use regarding varifolds. For a more detailed treatment, see Simon \cite{simon1984lectures}, Federer-Fleming \cite{federer2014geometric}, De Lellis \cite{DeL12}.

Let $U\subset \mathbb R^{n+1}$ be an open set. An integral varifold $V$ of dimension $n$ in $U$ is a pair $(N, \theta)$ such that  $N\subset U$ is an $n$-rectifiable set and $\theta:N\rightarrow \mathbb N\backslash \{0\} $ (multiplicity function) a Borel map. We associate a measure $\mu_N$ to $N$, defined on any Borel set $A\subset\mathbb R^{n+1}$ by
$$ \mu_N(A)=\int_{N\cap A} \theta d\mathcal H^n.$$\\
The varifold $N$ is said to be stationary if the first variation is zero. The density of the varifold in an $r$-ball is given by 
$$ \Theta_N(r, x)\coloneqq \frac{\mu_N( B_r(x))}{w_n r^n},$$
where $w_n$ is the volume of the unit $n$-ball. If $N$ is stationary, the monotonicity formula states that $r\rightarrow\Theta_N(r, x)$ is non-decreasing as a function of $r$. 

In general a stationary varifold can have singularities, and higher multiplicity, but with our mass bound in Theorem \ref{main:maintheo}, we will see that the varifold $N$ is actually a (multiplicity one) smooth submanifold in the annulus $B_1\setminus B_{1/2}$ (see Proposition~\ref{prop:FromAllard}). 
More precisely, by a minimal hypersurface in $B_1$, with boundary $\Sigma'$ we mean a stationary integral $n$-varifold $N$ in $B_1$ satisfying:
\begin{enumerate}
\item Writing $\mathrm{spt}\, N$ for the support of $\mu_N$, we have $\overline{\mathrm{spt}\, N} \cap \partial B_1 = \Sigma'$,
\item For any $p\in \Sigma'$ the varifold $N$ has density $\frac{1}{2}$ at $p$ in the sense that
\[ \lim_{r\to 0} \frac{\mu_N(B_r(p))}{r^n\omega_n} = \frac{1}{2},\]
where $\omega_n$ is the volume of the unit Euclidean ball in $\mathbb{R}^n$. 
\end{enumerate}

\subsection{Jacobi fields}
Suppose that $N$ is a smooth minimal hypersurface in an open set $U$. For a function $u$ on $N$ that is sufficiently small in $C^{2,\alpha}$, we consider the graph, $graph_N(u)$ of $u$ over $N$. We let $\mathcal{M}_N(u)$ denote the mean curvature of $graph_N(u)$. 
The mean curvature $\mathcal{M}_N(u)$ is a $2^{nd}$ order quasi-linear differential operator, of the form
$$\mathcal{M}_N(u)=a_N(x,u,\nabla u)^{ij}\nabla^2_{ij}u+b_N(x,u,\nabla u).$$
Let $\mathcal{L}_N(u)$ be the linearization of $\mathcal{M}_N$ at $0$, i.e., for $u:N\rightarrow \mathbb R$, 
$$\left. \mathcal{L}_N(u)=\frac{d}{dt}\right |_{t=0} \mathcal{M}_N(tu) .$$  In general, $\mathcal{L}_N$ is given by 
\begin{equation}\label{L_N}
\mathcal{L}_N=\Delta_N+|A_N|^2,
\end{equation}
where $A_N$ is the second fundamental form of $N^n\subset \mathbb R^{n+1}$, and $\Delta_N$ is the Laplacian operator on $N$. Functions in the kernel of $\mathcal{L}_N$ are called Jacobi fields on $N$. We define
$$Q_N(u)\coloneqq \mathcal{M}_N(u)-\mathcal{L}_N(u),$$
which gives the higher order terms of the mean curvature operator. 

\subsection{Quadratic cones}
Area-minimizing hypersurfaces $N^n$ with $n\leq 6$ are smooth by Simons~\cite{SimonJ}, and this inequality is sharp. The Simons cone $C(S^3\times S^3)$ was the first example of an area-minimizing minimal hypersurface with a singularity shown by Bombieri-De Giorgi-Giusti \cite{Bombieri1969}. More generally, quadratic cones are among the simplest area-minimizing cones and will be a key focus of the paper. 
\begin{definition}[Quadratic cones]
\label{prel:defn:Quadratic cones}
For $p, q\geq 1$ we define quadratic cones, also known as generalized Simons cones, as 
$$\mathcal C^{p, q}\coloneqq C(S^p\times S^q)=\{(x, y)\in\mathbb R^{p+1}\times \mathbb R^{q+1}: q|x|^2=p|y|^2\}\subset \mathbb R^{n+1},$$ where $n-1=p+q$. 
\end{definition}
The cone $\mathcal C^{p, q}$ is minimal for all $(p,q)$, and area-minimizing when $p+q>6$ or when $(p,q)=(3,3), (2,4), (4,2)$ (see Sim\~oes~\cite{Simoes74}, Davini~\cite{Davini}). For ease of notation throughout this paper, we will denote $\mathcal{C}\coloneqq \mathcal C^{p,q}$ for one of the area-minimizing quadratic cones. Since we are interested in the unit ball $B_1=\{x\in \mathbb R^{n+1}: |x| <1\}$, denote 
$$\mathcal{C}_1\coloneqq \mathcal{C}\cap B_1, \:\: \Sigma\coloneqq \mathcal C\cap \partial B_1 .$$

\subsection{Jacobi fields on $\mathcal{C}$}
\begin{definition}[Weighted H\"older norm on $\mathcal C$]
Let $u:\mathcal{C}\cap B_a\backslash B_b\rightarrow \mathbb{R}$ be a function on a region in $\mathcal{C}$. The \emph{weighted H\"older norm} of $u$ is 
\begin{equation*}
\|u(.)\|_{C^{k, \alpha}_{\delta}}\coloneqq \sup_{r > 0}\|r^{-\delta}u(r.)\|_{C^{k, \alpha}(\mathcal{C}\cap (B_1\backslash B_{1/2}))}.
\end{equation*}
We say that $u\in C^{k, \alpha}_{\delta}$ if  $\|u\|_{C^{k, \alpha}_{\delta}}$ is finite. We will only use $\delta = 1$, since $C^{2,\alpha}_1$ is the natural space of functions whose graphs over $\mathcal{C}$ can be realized as a surface. 
\end{definition}
We use the same notation for weighted spaces on other hypersurfaces such as $M_j^i$ near $\mathcal C$ were in annular regions where $M_j^i$ is graphical and close to the cone we can define a similar norm (see \cite[Chapter 4]{Marshall} for further discussion on weighted spaces).  

Using the second fundamental form of $\mathcal C$, note
$$\mathcal{L_C}=\Delta_\mathcal C+\frac{(  n-1)}{|x|^2}.  $$
For $x\in \mathcal C \backslash \{0\}$, we can write $x=rw$ where $r\in (0, \infty)$ and $w\in \Sigma$. Writing the above equation in polar coordinates, we get, 
$$\mathcal{L_C}=\frac{\partial^2}{\partial r^2}+\frac{n-1}{r}\frac{\partial }{\partial r}+\frac{1}{r^2}\mathcal L_\Sigma, $$
where $\mathcal L_\Sigma=\Delta_\Sigma+(n-1)$. Decomposing $L^2(\Sigma)$ using an orthonormal basis $\{ \phi_i\}_{i=1}^{\infty}$ and corresponding eigenvalues $\mu_1\leq \mu_2 \leq \dots \rightarrow \infty,$ such that
\begin{equation}\label{eigenvalue}
\mathcal{L}_{\Sigma} \phi_i=-\mu_i \phi_i,
\end{equation}
and, $\langle\phi_i, \phi_j\rangle=\delta_{ij}$.
Given a function $u$ on (a subset of) $\mathcal{C}$, we can write $u$ using the above basis in the form
$$u(r\omega)=\sum_{i=1}^{\infty}a_i(r) \phi_i(\omega). $$
Solving $\mathcal{L_C}u=0$, we get that Jacobi fields on $\mathcal C$ are of the form, 
\begin{equation}\label{eqn: u}
    u(rw)=\sum_{i=1}^\infty\left(c_i^+r^{\gamma_i^+}+ c_i^-r^{\gamma_i^-} \right)\phi_i(w),
\end{equation}
where
$$ \gamma^\pm_i=\frac{2-n\pm \sqrt{(n-2)^2+4\mu_i}}{2}.$$
Using the spectrum of spheres, one can explicitly compute the values of $\gamma_i$'s for quadratic cones (see \cite{SimonSolomon}).  For an area minimizing quadratic cones we define $\gamma=\gamma_1^+$, that satisfies $\frac{2-n}{2}<\gamma < -1$.  For example, for the Simons cone $\mathcal C^{3, 3}$ $\gamma_1^+=-2$.
Also note that $\phi_i$ (see \cite{SimonSolomon}),  $i=1, 2, 3$  geometrically corresponds to Hardt-Simon foliations, translations and rotations. 

\subsection{Hardt-Simon Foliation}
Given the quadratic cone $\mathcal{C}= \mathcal{C}^{p,q}$ with link $\Sigma$, notice that $\mathbb{R}^{n+1} \backslash \mathcal{C}$ divides $\mathbb R^{n+1}$ into two connected regions $E_{+}$, $E_-$, where 
\begin{align*}
E_+ &=\{(x, y)\in \mathbb R^{p+1} \times \mathbb R ^{q+1}: q|x|^2>p|y|^2 \},\\
E_- &=\{(x, y)\in \mathbb R^{p+1} \times \mathbb R ^{q+1}: q|x|^2<p|y|^2 \}.
\end{align*}
Choose an orientation on $\mathcal{C}$ such that the normal $\nu_{\mathcal C}(x,y)$ points in $E_+$. Hardt-Simon showed the following.

\begin{theorem}[\cite{SimonHardt}] \label{thm:HardtSimon}
For an area-minimizing quadratic cone $\mathcal C=\mathcal C^{p,q}$, there exists a connected oriented smooth minimal hypersurface $S_+\subset E_+$. This surface is unique up to scaling, that is if $S\subset E_+$ is a connected smooth minimal hypersurface, then $S=\lambda S_+$ for some $\lambda>0$.  Moreover the hypersurfaces $\lambda S_\pm$ for $\lambda > 0$ foliate $E_\pm$.
\end{theorem}
Hardt-Simon had proven this more generally for regular area-minimizing cones, and recently the existence (but not the uniqueness) of $S_+$ was shown for general area-minimizing cones by Wang {\cite{Wang24}}.

We define
$$ S_{\lambda}=\left\{\begin{array}{ccc}
      \lambda S_+ \: &\lambda>0, \\
      \mathcal{C} \: &\lambda=0, \\
      |\lambda| S_- \: &\lambda<0.
 \end{array} \right. $$

Since minimizing quadratic cones are strictly minimizing in the sense of Hardt-Simon~\cite{SimonHardt}, see in particular loc. cit. Remark 3.3 (1), the surfaces $S_{\pm}$ correspond to the Jacobi field $r^{\gamma}$, i.e, $S_{\pm}$ are graphical over $\mathcal C$ outside of a ball $B_{R_0}$ for some $R_0>0$:
$$ S_{\pm} \backslash B_{R_0}=\mbox{graph}_{\mathcal{C}}\{ \pm r^\gamma+O(r^{\gamma -\epsilon})\},$$
where we have normalized $S_{\pm}$ such that the coefficient of $r^\gamma$ is $\pm 1$. By scaling, we can use this to see that, for $\lambda>0$ we have
\begin{equation}\label{graphS_lambda}
S_{\pm \lambda}\backslash B_{\lambda R_0}=\mbox{graph}_{\mathcal{C}}\{ \pm \lambda^{-\gamma +1}  r^{\gamma}+O(r^{\gamma -\epsilon})\}.
\end{equation}

We will also need the following global rigidity result due to Simon-Solomon.
\begin{theorem}[\cite{SimonSolomon}, \cite{Mazet_2017}]\label{Theo: Simon-Solomon}
Let $N$ be a minimal hypersurface in $\mathbb R^{n+1}$ with tangent cone at infinity given by the quadratic cone  $\mathcal{C}$. Then, up to translations and rescalings, $N=\mathcal{C}$, $S_+$, or $S_-$.   
\end{theorem}

The fact that quadratic cones are strictly stable and strictly minimizing implies the following. 
\begin{lemma} \label{prelim:GammaBound}
We have the following, 
\begin{enumerate}
    \item Let $w$ be a Jacobi field on $S_\lambda$ for $\lambda\in \mathbb R\setminus\{0\}$ such that 
$$ |w(x)|<C|x|^{\frac{2-n}{2}}$$
for $|x| > 1$.  Then $w=0$.

\item Let $w$ be a Jacobi field on $\mathcal C_1 = \mathcal{C} \cap B_1$ with $w= 0$ on $\Sigma$ and $|w(x)|<C|x|^{\frac{2-n}{2}}$ for $|x| < 1$. Then $w=0$  on $\mathcal C_1$.
\end{enumerate}

\end{lemma}
\begin{proof}
(1) Without loss of generality we can work with $S_+$. The family $S_\lambda$ for $\lambda \to 1$ generates a positive Jacobi field $w_0$ on $S_+$. Since $\mathcal{C}$ is strictly minimizing, we have that $w_0 = c_1|x|^\gamma + O(|x|^{\gamma - \epsilon})$ for some $c_1, \epsilon > 0$ (see \cite[Theorem 3.2]{SimonHardt}), and recall that $\gamma > \frac{2-n}{2}$. 

Since $|w(x)|\leq C|x|^{\frac{2-n}{2}}$, there exists a $c > 0$  such that 
\begin{eqnarray*}
     w(x)\leq c w_0(x),
\end{eqnarray*}
for $x\in S_+$, and we choose $c_+$ to be the infimum of all such $c$. If $c_+ > 0$, then since $w$ decays faster than $w_0$ at infinity, there must be a point $x_0$ such that $w(x_0) = c_+ w_0(x_0)$. In this case the maximum principle implies that $w=c_+ w_0$, which contradicts that $w$ decays faster than $w_0$. It follows that $c_+ =0$, so $w \leq 0$. Similarly we can show that $w\geq 0$, so we have $w=0$. 

(2) Using (\ref{eqn: u}) note that on $\Sigma$ we get that $c_i^++c_i^-=0$. Since, $\gamma_1^+> \frac{2-n}{2}> \gamma_1^-$ we get that $c_i^-=0$ for all $i$ and hence we get that $u\equiv 0$ on $\mathcal C_1$.
\end{proof}

\subsection{Allard's theorem}

Allard's regularity theorem is a key tool in understanding local regularity of stationary varifolds. It essentially states that if in a ball a stationary integral varifold $V$ is close in the sense of measures to a hyperplane, then in a smaller ball $V$ can be represented by a $C^{1, \alpha}$ embedded hypersurface. See \cite{simon1984lectures}, \cite{allard1972first}, \cite{DeL12} for a more detailed version for varifolds with bounds on its mean curvature. Similarly, Allard~\cite{Allardboundary} (see also \cite{Bourni}) proved a boundary regularity theorem: if an integral varifold $V$ is stationary outside of a codimension-two $C^{1,\alpha}$ submanifold $\Sigma$, and in small balls $B_r(x_0)$ with $x_0\in \Sigma$, $V$ is close to a half plane in the sense of measures, then $V$ is a $C^{1,\alpha}$-graph in the balls $B_{r/2}(x_0)$.

In our setting, these regularity theorems imply the following. 
\begin{prop}\label{prop:FromAllard}
    Let $c_0, r_0 > 0$ be small. There exists an $\epsilon_0 > 0$ depending on $c_0, r_0$, with the following property. Let $\Vert g\Vert_{C^{2,\alpha}(\Sigma)} < \epsilon_0$ and suppose that $N$ is a stationary integral varifold in $B_1$ with boundary $\Sigma'$ given by the spherical graph of $g$ over $\Sigma$. In addition suppose that $\mu_N(B_1) < \frac{3}{2} \mu_\mathcal C (B_1)$. 
    Then in the annulus $B_1\setminus B_{r_0}$ the varifold $N$ is the spherical graph of a function $v$ over $\mathcal{C}$ with $\Vert v\Vert_{C^{2,\alpha}} < c_0$. 
\end{prop}
\begin{proof}
We argue by contradiction. Suppose that for given $c_0, r_0 > 0$ we have a sequence $N^i$ whose boundaries are the graphs of $g_i$ over $\Sigma$ with $\Vert g_i\Vert_{C^{2,\alpha}} \to 0$. 

Using the family of hypersurfaces $S_\lambda$ as barriers, we see that $N^i$ is contained between $S_{\pm \lambda_i}$ for a sequence $\lambda_i\to 0$. Up to choosing a subsequence we can assume that the $N^i$ converge to a stationary varifold $N^\infty$, which is necessarily supported on $\mathcal{C}\cap B_1$. The bound on the total mass together with the constancy theorem imply that $N^\infty = \mathcal{C}\cap B_1$. Using Allard's interior and boundary regularity results~\cite{allard1972first, Allardboundary} we find that for the given $r_0$ and any small $\delta > 0$, for sufficiently large $i$ the varifold $N^i$ is the graph of a function $v_i$ over $\mathcal{C}$ on $B_1\setminus B_{r_0}$, with $\Vert v_i\Vert_{C^{1,\alpha}} < \delta$. The higher regularity estimate $\Vert v_i\Vert_{C^{2,\alpha}} < c_0$ follows from this. 
\end{proof}

We will also need the following theorem (see e.g. \cite[Theorem 13.1]{Edelen2024}), stating that for large $i$, the hypersurfaces $M^i_j$ are close to the cone $\mathcal{C}$ at all scales at which the density is close to that of $\mathcal{C}$. 

\begin{theorem}\label{prel:thm:Density}
      Given $\epsilon>0$, there exist $\delta (\epsilon, \mathcal C)$, $\tau (\epsilon, \mathcal C)$ with the following properties. Suppose that $V$ is a multiplicity one stationary integral $n$-varifold in $B_1$ and for some $r_0 < 1$ we have
 \begin{eqnarray*}
        \mbox{spt}\, V\cap B_1\backslash B_\frac{1}{2}&=& graph_\mathcal C (u),\:\: \|u\|_{C^2(B_1\setminus B_{1/2})}\leq \delta\\
        \Theta_V(1,0)&\leq& \Theta(\mathcal C)+\tau,\\ 
        \Theta_V\left(\frac{r_0}{4},0\right)&\geq& \Theta(\mathcal C)-\tau.
    \end{eqnarray*}
  Then $u $ can be extended such that, 
\begin{equation}
\mbox{spt}\,V\cap B_1\backslash B_{r_0}=graph_\mathcal C (u)\cap B_1\backslash B_{r_0}, \:\:\: \|u\|_{C^2_1(B_1\setminus B_{r_0})}\leq \epsilon
\end{equation}
\end{theorem}

\subsection{Three Annulus Lemma}
A final important ingredient in our argument is the three annulus lemma. To state this, we define the following scaled $L^2$-norm for a function $u$ on an annular region on $M_1^i$:
$$ \|u\|^2_{\rho_0, k}\coloneqq \int_{M_1^i\cap B_{\rho_0^k}\backslash B_{\rho_0^{k+1}}}|u|^2 r^{-n},$$
where $k$ is an integer, $\rho_0\in (0,1)$ and $r$ denotes the distance from the origin. We use the same notation for norms of functions over other hypersurfaces such as $M_2^i$ and its rescalings, or over $\mathcal{C}$. 

We have the following three annulus lemma, see e.g. Simon~\cite{simon1983isolated}, or \cite{székelyhidi2021minimalhypersurfacescylindricaltangent} for a similar statement. 

\begin{lemma}[Three Annulus Lemma]\label{prelim:ThreeAnnulusLemma}
There exist $\epsilon, \rho_0 > 0$ depending on $\mathcal{C}$ with the following property. Suppose that $N_1, N_2$ are minimal hypersurfaces in the annulus $B_{\rho_0^k}\setminus B_{\rho_0^{k+3}}$ for an integer $k$. Suppose that $N_j$ is the spherical graph of $v_j$ over $\mathcal{C}$ on this annulus, where $\Vert v_j\Vert_{C^{2,\alpha}_1} < \epsilon$, and write $N_2$ as the spherical graph of $u$ over $N_1$. If,
$$\|u\|_{\rho_0, k+1}\geq \rho_0^{\frac{2-n}{2}} \|u\|_{\rho_0, k}, $$
then we have
$$ \|u\|_{\rho_0, k+2}\geq \rho_0^{\frac{2-n}{2}} \| u\|_{\rho_0, k+1}.$$
\end{lemma}
\begin{proof}

The proof is essentially the same as that of Simon \cite[Part II Lemma 3.3]{simon1983isolated}. We briefly sketch the argument, for $k=0$. We argue by contradiction, so suppose that we have sequences $N^i_1, N^i_2$ satisfying the assumptions for a small $\rho_0 > 0$ and $\epsilon = \epsilon_i \to 0$. We write $N^i_2$ as the graph of $u_i$ over $N^i_1$ on the annulus $B_1\setminus B_{\rho_0^3}$, and it follows that $\Vert u_i\Vert_{C^2_1(B_1\setminus B_{\rho_0^3})} \to 0$. We suppose in addition that for all $i$ we have

\[
    \|u_i\|_{\rho_0, 1}\geq \rho_0^{\frac{2-n}{2}} \|u_i\|_{\rho_0, 0},
\]
but at the same time 
\[ 
    \|u_i\|_{\rho_0, 2} <  \rho_0^{\frac{2-n}{2}} \|u_i\|_{\rho_0, 1}.
\]
From the above two equations we have $\|u_i\|_{\rho_0, 1}\neq 0$, so we can define
$$ \tilde u_i\coloneqq \frac{u_i}{\|u_i\|_{\rho_0, 1}}.$$
Note that by our assumption we have a uniform $L^2$-bound for $u_i$ on the annulus $B_1\setminus B_{\rho_0^3}$. 
Using that $N^i_1, N^i_2$ are minimal hypersurfaces, we obtain uniform derivative estimates for $u_i$ on any compact subset of $N^i_1$. Using the convergence $N^i_1 \to \mathcal{C}$ on the annulus $B_1\setminus B_{\rho_0^3}$, we can extract a limit $u_i \to w$, where $w$ is defined on the annulus $B_1\setminus B_{\rho_0^3}$ in $\mathcal{C}$. The convergence $u_i\to w$ is smooth on compact subsets of the annulus, so $w$ is not identically zero. In addition $\mathcal{L}_\mathcal{C} w=0$, and we have
\begin{eqnarray*}
    \| w\|_{\rho_0, 1} &\geq& \rho_0^{\frac{2-n}{2}} \|w\|_{\rho_0, 0}, \\
    \|w\|_{\rho_0, 2} &\leq& \rho_0^{\frac{2-n}{2}}\|w\|_{\rho_0, 1}.
\end{eqnarray*}
This is a contradiction to the three annulus lemma for Jacobi fields on $\mathcal C$, since there are no homogeneous Jacobi fields of growth rate equal to $\frac{2-n}{2}$. 
\end{proof}

\section{Proof of theorem \ref{thm:main}} \label{sec:mainproof}

We  prove Theorem~\ref{thm:main}  by contradiction, first assuming a mass bound that rules out issues with multiplicity. For this we show the following, the proof of which will take up most of this section. 

\begin{theorem}\label{main:maintheo}
Let $M_1^i$ and $M_2^i$ be minimal hypersurfaces in $B_1\subset \mathbb R^{n+1}$ such that $\Theta_{M_j^i}(1,0)< \frac{3}{2}\Theta(\mathcal C)$, and for a sequence of functions $g_i$ on $\Sigma$ with $\Vert g_i\Vert_{C^{2,\alpha}}\to 0$ we have the boundary conditions
\[ M_1^i \cap \partial B_1=M_2^i\cap \partial B_1=graph_{\Sigma}\{ g_i\}. \]
Then $M_1^i=M_2^i$ for large $i$.
\end{theorem}

First, it follows from Proposition~\ref{prop:FromAllard} that 
we have a sequence $\epsilon_i \to 0$ such that in the annulus $B_1\backslash B_{\epsilon_i}$, the $M^i_j$ are the graphs of $u^i_j$ over $\mathcal{C}$ where $\Vert u_j\Vert_{C^{2,\alpha}_1} < \epsilon_i$. 
Theorem \ref{prel:thm:Density} gives a more refined result, stating  that at scales at which $M_j^i$ is close to the cone in terms of density, it will also be graphical over the cone. With this motivation in mind, we fix a small $\tau > 0$ (to be chosen later), and for $j=1,2$ we define the following:  for $x\in B_{1/2}(0)$ we let
\[ 
r_j^i(x)\coloneqq\inf \left\{r < 1/2\,:\,\Theta(\mathcal C)-\tau\leq\Theta_{M_j^i}\left(\frac{r}{4}, x\right)\leq\frac{3}{2}\Theta(\mathcal C) \right\}, 
\]
and we let 
\begin{equation}\label{main:r_1}
r_j^i = \inf_{x\in B_{1/2}} r^i_j(x). 
\end{equation}

Note that for $j=1,2$ we have $M_j^i\rightarrow \mathcal C$ as $i\to\infty$. Therefore for any $x\in B_{1/2}(0)$ and $r < \frac{1}{2}$ we have
\begin{equation}\label{eq:Thetalim}
    \lim_{i\rightarrow\infty}\Theta_{M_j^i} (r, x)=\Theta_{\mathcal{C}}(r,x).
\end{equation}
It follows that for large $i$ the sets whose infimum we take above are nonempty, and also $r^i_j\to 0$ as $i\to\infty$. 
Let us also write $x^i_1, x^i_2$ for points at which these infima are achieved. From \eqref{eq:Thetalim} we  have that
$x^i_1, x^i_2 \to 0$. We can view $x^i_1, x^i_2$ as the points around which the $M^i_j$ look ``most conical", and the $r^i_j$ as the smallest scale at which they still look conical. Note that it is possible that $r^i_j = 0$. 

Switching $M^i_1$ with $M^i_2$ along a subsequence if necessary, we can assume that $r_1^i\leq r_2^i$ for all $i$. Let us translate both sequences by the same quantities so that the $M_1^i$ are ``centered" at the origin, i.e, let us define 
$$ \hat M_1^i\coloneqq M_1^i-x^i_1$$ 
and 
$$\hat M_2^i=M_2^i-x^i_1.$$ 
After this translation we have 
$$r_1^i=\inf\left\{r < 1/2: \Theta(\mathcal C)-\tau\leq \Theta_{\hat M_1^i}\left(\frac{r}{4}, 0\right)\leq \frac{3}{2}\Theta(\mathcal C)\right\}, $$
and we define 
$$ \tilde r_2^i\coloneqq \inf\left\{r < 1 /2:\Theta(\mathcal C)-\tau \leq\Theta_{\hat M_2^i} \left(\frac{r}{4}, 0\right)\leq \frac{3}{2}\Theta(\mathcal C) \right\}.$$

Observe that $\tilde r_2^i$ need not be the same as $r_2^i$, but we have
$$ 0\leq r_1^i\leq r_2^i \leq \tilde r_2^i.$$
We have translated so that $\hat M_1^i, \hat M_2^i$ are minimal hypersurfaces in the ball $B_1(x^i_1)$, with the same boundary condition on $\partial B_1(x^i_1)$. Recall  that $x^i_1 \to 0$ as $i\to\infty$. 

\begin{lemma}\label{lem:graph}
Given $\epsilon > 0$ we can choose $\tau = \tau(\epsilon, \mathcal{C}) > 0$ above, so that for sufficiently large $i$ (depending on $\epsilon$ as well) on the annular region $B_1(x_1^i)\setminus B_{\tilde{r}^i_2}$ both $\hat{M}^i_1, \hat{M}^i_2$ are graphs of $v_1, v_2$ over $\mathcal{C}$, and $\Vert v_j\Vert_{C^{2,\alpha}_{1}} < \epsilon$. If $\tilde{r}_2^i=0$, then the graphicality extends down to the origin.
\end{lemma}
\begin{proof}
 This follows from Proposition~\ref{prop:FromAllard} and Theorem~\ref{prel:thm:Density}. Let $\epsilon$ be small. Since $x_1^i\to 0$ and $M^i_j \to \mathcal{C}\cap B_1$, for sufficiently large $i$ we can ensure that $\hat M_1^i$ is the graph of $v_1^i$ over $\mathcal C$ in $B_{1-\epsilon}\backslash B_{1/4}$ with $\|v_1^i\|_{C^2(B_{1/2}\backslash B_{1/4})}\to 0$. From  Theorem \ref{prel:thm:Density} and the definition of $\tilde{r}_2^i$ we obtain the required graphicality of $\hat{M}^i_{1,2}$ over $\mathcal{C}$ on $B_{1-\epsilon} \setminus B_{\tilde{r}_2^i}$. The graphicality near the boundary of $B_1(x^i_1)$ follows from Proposition~\ref{prop:FromAllard}.
\end{proof}

In particular on the annulus $B_1(x^i_1) \setminus B_{r^i_2}$ we can view $\hat{M}^i_2$ as the graph of a function $u_i$ over $\hat{M}^i_1$ for large $i$. 
We choose $\epsilon, \rho_0$ small such that the three annulus Lemma~\ref{prelim:ThreeAnnulusLemma} applies to this function, and then choose $\tau(\epsilon, \mathcal C)$ accordingly and define $\tilde r_2^i$ (and $r_2^i$ and $r_1^i$) with this $\tau$. We can assume that $\rho_0$ is small such that  $B_{\rho_0} \setminus B_{\rho_0^2}$ is contained in $B_{1-\epsilon}\setminus B_{\tilde{r}^i_2}$. 

Let us define
\[ k_i\coloneqq \inf\{ k\,:\,  \|u_i\|_{\rho_0,  k+1}\geq \rho_0^{\frac{2-n}{2}}\|u_i\|_{\rho_0,k}, \text{ where } k\geq 1, \rho_0^{k+2} > \tilde{r}^i_2 \}. \]
Here  $\|u_i\|_{\rho_0, k}$ are as in Lemma~\ref{prelim:ThreeAnnulusLemma}, and if no suitable $k$ exists, we set $k_i=\infty$. In other words, 
$\rho_0^{k_i}$ is the scale at which $\hat{M}^i_2$ starts diverging  from $\hat{M}^i_1$ at rate at least $(2-n)/2 < \gamma$ as $r\to 0$. Note that by Lemma~\ref{prelim:ThreeAnnulusLemma} we have 
\begin{equation} \label{eq:uigrowth}
\|u_i\|_{\rho_0,  k+1}\geq \rho_0^{\frac{2-n}{2}}\|u_i\|_{\rho_0,k}, 
\end{equation}
for all $k\geq k_i$, as long as $\rho_0^{k+3} > \tilde{r}^i_2$.

We will next show that the sequence of $k_i$ can be assumed to be uniformly bounded as $i\to\infty$. 
\begin{lemma}\label{lemma:k_i bounded}
    If the $M^i_1, M^i_2$ are distinct, then 
    there is a constant $K > 0$ such that $k_i < K$ for all sufficiently large $i$. 
\end{lemma}
\begin{proof}
We argue by contradiction. Suppose that $k_i\rightarrow \infty$ ($\rho_0^{k_i}\rightarrow 0$) or $k_i=\infty$ along a subsequence. In either scenario we have a sequence of positive integers $l_i\rightarrow \infty$ such that 
For all $0\leq k<l_i$, 
\begin{equation}\label{eqn: M_j distinct}
 \|u_i\|_{\rho_0, k+1}< \rho_0^{\frac{2-n}{2}}\|u_i\|_{\rho_0, k}<\rho_0^{k(\frac{2-n}{2})}\|u_i\|_{\rho_0, 1}.
\end{equation}
We assume $u_i\neq 0$ and define $u_i^o:\hat M_1^i\cap B_1(x_1^i)\backslash B_{\rho_0^{k_i}}\rightarrow \mathbb R$ as
$$ u_i^o\coloneqq\frac{u_i}{\|u_i\|_{L^2(\hat M_1^i\cap B_1(x_1^i)\backslash B_{\rho_0^2})}}. $$

We claim that there exists a subsequence such that $u_i^o \rightarrow w$, where $w$ is a Jacobi field on $\mathcal C \cap B_1$, with $w=0$ on $\Sigma$ and $|w(x)| \leq C |x|^{\frac{2-n}{2}}$. Here the convergence is $C^{2,\alpha}$ on compact subsets of $\mathcal{C}\cap (B_1\setminus \{0\})$, up to the boundary $\partial B_1$, and in fact smooth on the interior. In order to show this, we will show that for large $i$ the $u^o_i$ satisfy uniform $C^{k,\alpha}$ estimates on any annulus $B_1(x^i_1)\setminus B_r$. Recall that from Proposition~\ref{prop:FromAllard} we know that as $i\to \infty$, the $\hat{M}^i_j$ restricted to any such annulus converge smoothly to $\mathcal{C}\cap (B_1\setminus B_r)$. 

To see this, note that
\[ \Vert u_i\Vert_{\rho_0,1} \leq C(\rho_0) \Vert u_i\Vert_{L^2(\hat{M}^i_1\cap B_1(x^i_1))\setminus B_{\rho_0^2}}. \]
From \eqref{eqn: M_j distinct} we then find that
\begin{equation}\label{eq:L2bound} \Vert u^o_i\Vert_{\rho_0, k+1} < C(\rho_0) \rho_0^{k(2-n)/2}, 
\end{equation}
for $k < l_i$. 

On any fixed annulus $B_1(x^i_1)\setminus B_r$ we can write the equation for $u_i$ as $\mathcal{M}_{\hat{M}_1^i}(u_i)=0$, and $u_i=0$ on $\partial B_1(x^i_1)$. In other words
\begin{equation}\label{eq:M=0} \mathcal{L}_{\hat{M}^i_1}(u_i) + Q_{\hat{M}^i_1}(u_i) = 0. 
\end{equation}
Since on the annulus $\Vert u_i\Vert_{C^{2,\alpha}}\to 0$ as $i\to\infty$, and $Q$ contains quadratic or higher order terms, it follows that the equation can be viewed as a homogeneous linear equation for $u_i^o$ with $C^\alpha$ coefficients (which depend on $u_i$). The interior and boundary Schauder estimates, together with the $L^2$ bounds \eqref{eq:L2bound}, imply that the $u_i^o$ satisfy uniform $C^{k,\alpha}$ bounds up to the boundary on any annulus $B_1(x^i_1)\setminus B_r$. In particular, up to choosing a subsequence, and letting $r\to 0$, we can extract a limit $w = \lim_{i\to\infty} u^o_i$ on $\mathcal{C}\cap B_1\setminus \{0\}$. Letting $i\to \infty$ in Equation \eqref{eq:M=0} implies that  $\mathcal{L}_{\mathcal{C}}(w)=0$, while the convergence along the boundary implies $w=0$ on $\Sigma$. 

Using the interior estimates, from \eqref{eq:L2bound} we can deduce that  $|w(x)|\leq C|x|^{\frac{2-n}{2}}.$ By Lemma~\ref{prelim:GammaBound} this leads to  $w\equiv 0$ which is a contradiction. 
\end{proof}

We can therefore assume from now that $k_i < K$ for all $i$. 
\begin{lemma}
    If the $M^i_1, M^i_2$ are distinct, then 
    there cannot be infinitely many $i$ such that $\tilde{r}^i_2=0$.
\end{lemma}
\begin{proof}
    If $\tilde{r}^i_2 = 0$, then we have that both $\hat M^i_1$ and $\hat M^i_2$ are graphical over $\mathcal{C}$ in $B_{1-\epsilon}$. In particular $\hat M^i_2$ is the graph of $u_i$ over $\hat M^i_1$ on $B_{1-\epsilon}$. Moreover we must have $|u_i(x)| < \epsilon |x|$ for all $|x| < 1-\epsilon$, for sufficiently large $i$. This contradicts the fact that we have the growth property \eqref{eq:uigrowth} for all $k > K$, unless $u_i=0$.  
\end{proof}

Since $\tilde{r}^i_2 \to 0$, and $k_i < K$, we now have
$(\tilde{r}^i_2)^{-1}\rho_0^{k_i}
\rightarrow \infty$.
Let $N_i$ be the largest positive integer such that, 
$$\rho_0^{N_i}\geq \tilde r_2^i > \rho_0^{N_i+1}.$$
Applying \eqref{eq:uigrowth}, we know that for any $k > K$ with $k\leq N_i-1$ we have
\begin{equation}\label{eq:uigrowth2}
\|u_i\|_{\rho_0,k+1}\geq \rho_0^{\frac{2-n}{2}} \|u_i\|_{\rho_0, k}. 
\end{equation}

We rescale both $\hat M_1^i$, $\hat M_2^i$ by $\rho_0^{N_i}$, i.e we define
$$ \widetilde M_j^i\coloneqq\frac{1}{\rho_0^{N_i}}\hat{M_j^i}, \quad j=1, 2. $$
We have the following. 
\begin{prop}\label{prop:Milimit} There are minimal hypersurfaces $\widetilde M_1^\infty, \widetilde M_2^\infty \subset \mathbb R^{n+1}$ such that for a subsequence as $i\to\infty$, we have
\begin{eqnarray*}
    \widetilde M_j^i \rightarrow \widetilde M_j^\infty, \: j=1, 2,
\end{eqnarray*}
as varifolds.
In addition $\widetilde M_j^\infty=a_j+Q_j(S_{\lambda_j})$ for some $ (a_j, Q_j, \lambda_j) \in \mathbb R^{n+1}\times SO(n+1)\times \mathbb R$.
The convergence outside of $B_1$ is in the locally smooth sense. In addition if $\lambda_1=0$, then $a_1=0$, that is $\widetilde M^\infty_1$ cannot be a non-trivial translation of a cone. 
\end{prop}
\begin{proof}
From Lemma~\ref{lem:graph}, $\hat M_1^i$ is the graph of $v_1^i$ over $\mathcal C$ in $B_{1-\epsilon}\backslash B_{\rho_0^{N_i}}$ and $\|v_1^i\|_{C^{2, \alpha}_1}\leq \epsilon $. Rescaling this we get that $\widetilde M_1^i$ is the graph of $\widetilde v_1^i$ over $\mathcal C$ in $B_{\rho_0^{N_i}(1-\epsilon)}\backslash B_1$ and $\|\widetilde v_1^i\|_{C^{2,\alpha}_1}\leq \epsilon$. Since the mass of $\widetilde M_1^i$ is bounded in any ball, independently of $i$, Allard's compactness theorem implies that along a subsequence we have a varifold limit $\widetilde M_1^i \rightarrow \widetilde M_1^\infty$ and $\widetilde M_1^\infty$ is a graph of $v_1^\infty$ over $\mathcal C\backslash B_1$ with $\|v_1^\infty\|_{C^{2, \alpha}_1}\leq \epsilon$.
The tangent cone $\mathcal{C}'$ at infinity of $\widetilde M_1^\infty$ is then also $\epsilon$-graphical over $\mathcal{C}$, and so, using the integrability of $\mathcal{C}$ we have that $\mathcal{C}' = Q\mathcal{C}$ for a rotation $Q$, for sufficiently small $\epsilon$. Simon-Solomon's Theorem~\ref{Theo: Simon-Solomon} implies that 
$\widetilde M_j^\infty=a_j+Q_j(S_{\lambda_j})$ for some $ (a_j, Q_j, \lambda_j) \in \mathbb R^{n+1}\times SO(n+1)\times \mathbb R$. 

If $\lambda_1=0$, then $\widetilde M^\infty_1$ is a translation of a cone by $a_1$. That $a_1=0$ follows from the way we chose $x^i_1$ to achieve the infimum in \eqref{main:r_1}. Indeed, we have
\[ \Theta_{\widetilde M^\infty_1}(r, a_i) = \Theta(\mathcal{C}) \]
for all $r > 0$. 
\end{proof}

\begin{prop} \label{prop: distinct} 
    We  have $\widetilde M^\infty_1= \widetilde M^\infty_2$. 
\end{prop}
\begin{proof}
Let us suppose that on the contrary, we have $\widetilde M_1^\infty\neq \widetilde M_2^\infty$. Note that outside of $B_1$, the hypersurface $\widetilde M^\infty_2$ is the graph of a function $U$ over $\widetilde M^\infty_1$, and moreover after rewriting the growth condition \eqref{eq:uigrowth2} for $U$ implies that for $k<-1$, 
\begin{equation} \label{proof: U Decay Eqn}
\Vert U\Vert_{\rho_0, k-1} \leq \rho_0^{\frac{n-2}{2}} \Vert U\Vert_{\rho_0, k}.
\end{equation}
It follows that $|U(x)| \leq C|x|^{\frac{2-n}{2}}$ for $|x| > 1$.  We can assume that $U\not\equiv 0$, since otherwise $\widetilde{M}^\infty_1 = \widetilde{M}^\infty_2$. 

Recall that $\widetilde M_j^\infty=a_j+Q_j(S_{\lambda_j})$. Translating by $a_1$ and rotating by $Q_1$ we find that for some $a, Q$, the hypersurface $N := a + Q(S_{\lambda_2})$ is the graph of a function $V$ over $S_{\lambda_1}$ outside of the ball $B_2$, and $V$ has the decay property $|V(x)| \leq C|x|^{\frac{2-n}{2}}$ for $|x| > 2$. We derive a contradiction from this using the surfaces $S_\Lambda$. 

For each $\Lambda \not=\lambda_1$ the hypersurface $S_\Lambda$ is the graph of a function $v$ over $S_{\lambda_0}$ outside of a large ball, where $|v(x)| \leq C_\Lambda |x|^\gamma$. Since $\gamma > \frac{2-n}{2}$, it follows that for any $\Lambda > \lambda_1$ the hypersurface $N$ lies on the negative side of $S_\Lambda$ near infinity. Similarly, for $\Lambda < \lambda_1$, $N$ lies on the positive side of $S_\Lambda$ near infinity. 

Let us define $\Lambda_+$ to be the infimum of all $\Lambda > \lambda_1$ such that $N$ lies on the negative side of $\Lambda$. If $\Lambda_+ > \lambda_1$, then $N$ must touch $S_{\Lambda_+}$ at some point. From the strong maximum principle (either Solomon-White~\cite{SolomonWhite} or Wickramasekera~\cite{Wickramasekera}) we find that $N=S_{\Lambda_+}$, which is a contradiction, since then $N$ does not decay towards $S_{\lambda_1}$ at rate $|x|^{(2-n)/2}$. It follows that $N$ lies on the negative side of any $S_\Lambda$ with $\Lambda > \lambda_1$. Arguing similarly with $\Lambda < \lambda_1$ we find that in fact $N=S_{\lambda_1}$, which contradicts the assumption that $\widetilde M^\infty_1 \not= \widetilde M^\infty_2$. 
\end{proof}

We can finally assume that $\widetilde M_1^\infty=\widetilde M_2^{\infty }$. The proof of Theorem~\ref{main:maintheo} is completed by the following. 
\begin{prop}\label{prop:same limit}
    Suppose that $\widetilde M_1^\infty = \widetilde M_2^\infty$. Then $M^i_1=M^i_2$ for sufficiently large $i$. 
\end{prop}
\begin{proof}
By the choice of $N_i$ and by the monotonicity formula we have that
$$ \Theta_{\widetilde M_2^i}\left(\frac{\rho_0^{N_i+1}}{\rho_0^{N_i}}, 0\right) \leq \Theta(\mathcal C)-\tau .$$
As $i\rightarrow \infty$, this implies that
$$\Theta_{\widetilde M_1^\infty} (\rho_0, 0) \leq \Theta(\mathcal C)-\tau. $$
This implies that $\widetilde M^\infty_1$ cannot be a cone, and so by Proposition~\ref{prop:Milimit} we find that $\widetilde M^\infty_1$ is smooth. Namely $\widetilde M^\infty_1 = a_1 + Q_1(S_{\lambda_1})$ for some $\lambda_1\not=0$. 

Recall that outside of the unit ball, by the construction we know that $\widetilde M^i_2$ is graphical over $\widetilde M^i_1$. The fact that $\widetilde M^\infty_1$ is smooth, together with Allard's regularity theorem implies that for large enough $i$, $\widetilde{M}^i_2$ is the graph of a function $U_i$ over $\widetilde{M}^i_1$ inside the unit ball as well. 

Consider the normalized functions
\[ U_i^o(x)\coloneqq  \frac{U_i (x)}{\|U_i \|_{L^2(B_{\rho_0^{-2}})}}\]
over $\widetilde{M}^i_1$. Note that we have the decay estimate $|U_i^o(x)| \leq C|x|^{\frac{2-n}{2}}$ for $|x| > 1$. Similarly to the argument in the proof of Lemma~\ref{lemma:k_i bounded} above, we find that if $M^i_1\not= M^i_2$ along a sequence, then along a further subsequence the $U_i^o$ converge to a nonzero Jacobi field $w$ on $\widetilde M^\infty_1$. The convergence is smooth on compact sets, and $w$ satisfies $|w(x)|\leq C|x|^{\frac{2-n}{2}}$ for $|x| > 1$. 

By Lemma \ref{prelim:GammaBound} there is no such non-zero Jacobi field on $S_\lambda$, for $\lambda\not=0$. Therefore we get a contradiction. 
\end{proof}

This completes the proof of Theorem~\ref{main:maintheo}.
We finally prove Theorem~\ref{thm:main} by relaxing the assumption of the mass bound in Theorem~\ref{main:maintheo}. 

\begin{proof}[Proof of Theorem~\ref{thm:main}]
Let $\epsilon > 0$ be small. For $g$ with $\Vert g\Vert_{C^{2,\alpha}(\Sigma)} < \epsilon$, let us denote by $M^P_g$ the solution of the Plateau problem with boundary given by the graph $\Sigma_g$ of $g$ over $\Sigma$. Comparing with the union of $\mathcal{C}\cap B_1$ with the hypersurface defined by the region between $\Sigma$ and $\Sigma_g$, we find that for small $\epsilon$ we must have $\Vert M^P_g\Vert < \frac{3}{2}\Theta(\mathcal{C})$. It follows from Theorem~\ref{main:maintheo} that the solutions of the Plateau problem for such boundary values $g$ are unique. As a consequence these solutions vary continuously with $g$, in the sense that if $g_i \to g$ in $C^{2,\alpha}$, then $M^P_{g_i} \to M^P_g$ as varifolds. By Proposition~\ref{prop:FromAllard} we also know that these $M^P_g$ all have good graphicality over $\mathcal{C}$ on the annulus $B_1\setminus B_{1/2}$. Let us choose $\epsilon_0 > 0$ small enough such that the solutions of the Plateau problem satisfy these properties for $\Vert g\Vert_{C^{2,\alpha}} < \epsilon_0$. 

We can now implement the strategy outlined in the introduction to show that even without the mass bound, there is a unique minimal hypersurface with boundary $\Sigma_g$ for small enough $g$. Note first that by using the Hardt-Simon surfaces $S_\lambda$ as barriers, we know that if $|g| < \epsilon$, then any minimal hypersurface with boundary $\Sigma_g$ is contained between $S_{\pm \lambda(\epsilon)}$, where $\lambda(\epsilon)\to 0$ as $\epsilon\to 0$. In fact we can take $S_{\pm \lambda(\epsilon)} = M^P_{\pm \epsilon}$ (i.e. the minimal surfaces with constant boundary values).

Suppose now that $\Vert g\Vert_{C^{2,\alpha}} < \epsilon_0/3$, and let $M_{g}$ be any minimal hypersurface with boundary $\Sigma_{g}$. By the above discussion we know that $M_{g}$ lies on the negative side of $M^P_{\epsilon_0/3}$. At the same time, for $t > 2\epsilon_0/3$ we have $g+t > \epsilon_0/3$, so the Plateau problem solution $M^P_{g+t}$ lies on the positive side of $M^P_{\epsilon_0/3}$. We can now let $T_+$ denote the infimum of the $t > 0$ such that $M_g$ lies on the negative side of $M^P_{g+t}$. If $T_+ > 0$, then $M_g$ must have an interior contact point with $M^P_{g+T_+}$, using that the $M^P_{g+t}$ vary continuously with $t$. At the same time $M^P_{g+T_+}$ lies on the positive side of $M_g$. We can apply the strong maximum principle of Wickramasekera~\cite[Theorem 1.1]{Wickramasekera} with $V_1 = M^P_{g+T_+}$ and $V_2 = M_g$. Note that since $M^P_{g+T_+}$ is area minimizing, its singular set has dimension at most $n-7$. Since the supports of $V_1, V_2$ intersect, it follows that they are equal, but this is not possible since they have different boundary values, so we must have $T_+=0$, i.e. $M_g$ lies on the negative side of $M^P_g$. Similarly we find that $M_g$ lies on the positive side of $M^P_g$, so necessarily $\mathrm{spt}\, M_g = \mathrm{spt}\, M^P_g$. The boundary condition implies that $M_g$ has multiplicity one.   This completes the proof of uniqueness. 
\end{proof}

\printbibliography

\end{document}